\documentclass[oneside]{amsart}

\numberwithin{equation}{section}
\newcommand{\be}{\begin{equation}}
\newcommand{\ee}{\end{equation}}

\newcommand{\R}{\mathbb R}
\newcommand{\Z}{\mathbb Z}

\newcommand{\ep}{\varepsilon}

\renewcommand{\phi}{\varphi}
\newcommand{\pd}{\partial}

\newcommand{\co}{\colon}

\DeclareMathOperator{\vol}{vol}

\newtheorem{theorem}{Theorem}
\newtheorem{lemma}{Lemma}[section]
\newtheorem{proposition}[lemma]{Proposition}

\theoremstyle{remark}
\newtheorem{remark}[lemma]{Remark}

\theoremstyle{definition}
\newtheorem{definition}[lemma]{Definition}

\begin{document}


\title{Minimality of planes in normed spaces}

\author{Dmitri Burago}                                                          
\address{Dmitri Burago: Pennsylvania State University,                          
Department of Mathematics, University Park, PA 16802, USA}                      
\email{burago@math.psu.edu}                                                     
                                                                                
\author{Sergei Ivanov}
\address{Sergei Ivanov:
St.Petersburg Department of Steklov Mathematical Institute,
Fontanka 27, St.Petersburg 191023, Russia}
\email{svivanov@pdmi.ras.ru}

\thanks{The first author was partially supported                                
by NSF grant DMS-0905838.
The second author was partially supported by
RFBR grant 11-01-00302-a.}

\subjclass[2010]{52A21, 52A38, 52A40, 53C60}

\keywords{Busemann--Hausdorff surface area, convexity, ellipticity}

\begin{abstract}
We prove that a region in a two-dimensional affine subspace of a
normed space $V$ has the least 2-dimensional Hausdorff
measure among all compact surfaces with the same boundary.
Furthermore, the 2-dimensional Hausdorff area density admits a convex extension
to $\Lambda^2 V$.
The proof is based on a (probably)
new inequality for the Euclidean area of a convex centrally-symmetric polygon.
\end{abstract}

\maketitle

\section{Introduction}

The main purpose of this paper is to prove the following seemingly elementary fact. 
Consider a bounded region in a two-dimensional affine plane in a normed vector space.
Then the region has the least possible two-dimensional Hausdorff measure among all
compact two-dimensional surfaces with the same boundary. Even though the problem 
sounds almost silly, it stood open for over 50 years and still remains open in dimensions
greater than 2. 

We prove our result by showing the convexity of the density function
for the Busemann--Hausdorff surface area in normed spaces.
This contrasts with a result of Busemann, Ewald and Shephard \cite{BES}
who demonstrated that the density of the Holmes--Thompson (symplectic) surface area
may fail to be convex.

We first embed the problem in a more general set-up, borrowing
some preliminaries  from \cite{BI04}.
Here we consider various notions of $k$-dimensional surface areas,
although the rest of the paper is devoted
to the case $k=2$ and the Busemann--Hausdorff definition of area
in a normed space.

A common way to introduce a (translation-invariant) $k$-dimensional
surface area in $\R^n$ is to define its density
in every $k$-dimensional linear subspace. Namely for a
continuous function $A\co G(n,k)\to\R_+$, where $G(n,k)=G_k(\R^n)$
is the Grassmannian manifold of $k$-dimensional linear subspaces of $\R^n$,
one defines the associated surface area functional $Area_A$ by 
the obvious formula
$$
Area_A(S)=\int_S A(T_xS)\,dm(x),
$$
where $S$ is a smooth (more generally, Lipschitz) surface,
$m$ is the $k$-dimensional Euclidean surface area, and the tangent spaces
$T_xS$ are regarded as points in $G(n,k)$.

If $\R^n$ is equipped with a norm $\|\cdot\|$, then every subspace
$P\in G(n,k)$ becomes a $k$-dimensional normed space. 
In the Busemann--Hausdorff definition of area, the density $A(P)$
is defined so that the area of the norm's unit ball (in $P$)
equals the standard constant $\ep_k$ depending only on $k$,
namely $\ep_k$ is the Euclidean volume of the Euclidean unit ball in $\R^k$.
The resulting area density $A^{bh}\co G(n,k)\to\R_+$ has the form
$$
 A^{bh}(P)= \frac{\ep_k}{m_k(P \cap B)}
$$
where $B\subset\R^n$ is the unit ball of the norm $\|\cdot\|$
and $m_k$ is the $k$-dimensional Euclidean area. The corresponding 
surface area functional has a clear geometric meaning: for embedded surfaces,
it is just the $k$-dimensional Hausdorff measure of the surface as a subset
of the normed space.

\begin{remark}
Another commonly used notion is the Holmes--Thompson area
whose density $A^{ht}$ is given by
$$
A^{ht}(P)= \frac1{\ep_k} m_k((P \cap B)^*),
$$
where $(P \cap B)^*$ is the polar body to $P\cap B$
with respect to the Euclidean structure in $P$.
Note that $(P\cap B)^*$ is the orthogonal projection
of $B^*$ to $P$ where $B^*$ is polar to $B$ in $\R^n$.

These two definitions of surface area come from Finsler geometry.
A smooth immersed surface in a normed space naturally
acquires the induced structure of a Finsler manifold exactly the same way as a 
surface in Euclidean space gets a Riemannian structure.
The surface areas in question can be regarded as Finsler volumes in the induced
Finsler metric and belong to its intrinsic geometry.

In Riemannian geometry, Riemannian volume bears two main meanings. Geometrically,
it is the Hausdorff measure. Dynamically, it is the projection of the Liouville measure
from the unit tangent bundle. In Finsler geometry, there is no
notion of volume which would enjoy the two properties.
The Busemann--Hausdorff and Holmes--Thompson definitions  
inherit one of the properties of Riemannian volume each.
The Busemann--Hausdorff volume of a $k$-dimensional Finsler manifold $M$ equals
the $k$-dimensional Hausdorff measure of the Finsler metric (see \cite{Bu47}). 
The Holmes--Thompson volume of $M$ (see \cite{HT}) is the (normalised by a suitable constant)
symplectic volume of the bundle of the unit balls in $T^*M$.

The two Finsler volumes can be expressed by the following  coordinate formulas
(for $\Omega \subset M$ identified with a subset in
$\R^k$, that is in one chart).
For the Busemann--Hausdorff volume one gets
$$
\vol^{bh}(\Omega) = {\ep_k} \int_\Omega m_k(B_x)^{-1} \,dm(x),
$$
and the Holmes--Thompson volume is given by
$$
\vol^{ht}(\Omega) = \frac1{\ep_k} \int_\Omega m(B^*_x) \,dm(x),
$$
where $m_k$ is the coordinate Lebesgue measure, 
$B_x$ is the unit ball of the Finsler norm at $x$: $B_x=\{v \in T_xM: \Phi(v) \le 1\}$.
Here we identify $T_xM$, $T^*_xM$ and $\R^k$ since $M=\R^k$.
The normalizing coefficient $\ep_k$ makes the volume definitions
agree with the Riemannian one for Riemannian manifolds.
\end{remark}

In the above definition of area densities as functions on $G(n,k)$,
the Euclidean structure of $\R^n$ is irrelevant.
Here is an affine-invariant definition.
Let $V$ be an $n$-dimensional vector space,
$\Lambda^k V$ the $k$th exterior power of $V$,
and $GC_k(V)\subset\Lambda^k V$ the $k$-dimensional Grassmannian cone,
that is the set of all simple $k$-vectors.
(A $k$-vector $\sigma\in\Lambda^k V$ is called \textit{simple}
if it is decomposable: $\sigma=v_1\wedge\dots\wedge v_k$
for some $v_1,\dots,v_k\in V$.)
A (translation invariant) \textit{$k$-dimensional density} in $V$ is
a continuous function $A\co GC_k(V)\to\R_+$ which is symmetric and positively homogeneous,
that is $A(\lambda\sigma)=|\lambda| A(\sigma)$ for all $\lambda\in\R$, $\sigma\in GC_k(V)$.
For a Lipschitz surface $S\co M\to V$
(parametrized by a smooth $k$-dimensional manifold $M$),
the integral $\int_S A$ is defined in an obvious way.
We refer to this integral as the $k$-dimensional surface area
associated with $A$ and denote it by $Area_A(S)$.

In the case of the Busemann--Hausdorff area in a normed space $(V,\|\cdot\|)$,
the density $A^{bh}\co GC_k(V)\to\R_+$ is defined as follows.
For a simple $k$-vector $\sigma=v_1\wedge\dots\wedge v_k$,
the value $A^{bh}(\sigma)$ equals the $k$-dimensional Hausdorff measure
(with respect to the metric defined by $\|\cdot\|$) of
the parallelotope spanned by the vectors $v_1,\dots,v_k$.
It can be expressed by the formula
$$
 A^{bh}(v_1\wedge\dots\wedge v_k) = \frac{\ep_k}{m_k(L^{-1}(B))},
$$
where $B$ is the unit ball of $\|\cdot\|$ and $L\co\R^k\to V$
is a linear map given by $L(e_i)=v_i$ for the standard
basis $(e_1,\dots,e_k)$ of $\R^k$.
We abuse notation and write $A^{bh}(S)$ instead of $Area_{A^{bh}}(S)$
for the Busemann--Hausdorff area of a surface $S$.

For a Lipschitz chain
$S=\sum a_i S_i$, $S_i\co\Delta \to \R^n$, where each $\Delta$ is
a standard simplex, we define $Area_A(S)=\sum |a_i| Area_A(S_i)$. 
The coefficients $a_i$ can be taken from
$\Z$, $\R$, or $\Z_2:=\Z/2\Z$. In the case of $\Z_2$,
the absolute values $|a_i|$ are defined as follows:
$|a_i|=0$ if $a_i=0$ and $|a_i|=1$ otherwise.
Note that a two-dimensional chain over $\Z$ or $\Z_2$
can be parameterized by a manifold (which is oriented in the case of $\Z$). 

A density $A\co GC_k(V) \to \R$ is said to be
\textit{convex} if it can be extended to a convex function on the vector
space $\Lambda^k V$ of all $k$-vectors.
The functional $Area_A$ is said to be \textit{semi-elliptic} over $\R$, $\Z$, or $\Z_2$
 (see \cite{Alm})
if, whenever the boundary $\pd S$ of a chain $S$ over the respective ring
is equal to the
boundary of a $k$-disc $D$  embedded into an affine $k$-plane, one
has $Area_A(S)\geq Area_A(D)$.

It is rather obvious that convexity of $A$ implies semi-ellipticity
of $Area_A$ over $\R$ and $\Z$.
The converse is true over $\R$ but in general may fail over $\Z$ (see \cite{BI04}).

For the Busemann-Hausdorff surface area density Busemann \cite{Bu49}
proved that it is convex in co-dimension one, that is if $\dim V=k+1$,
and left the general case as a conjecture (see e.g.\ \cite[p.~37]{BS60}, \cite[p.~180]{Bu61}
or \cite[p.~310, Problem 7.7.1]{T}). 
The main result of this paper is a proof of this conjecture for the 2-dimensional 
Busemann-Hausdorff surface area:

\begin{theorem}\label{main}
In every finite-dimensional normed space $V$, the two-dimensional Busemann--Hausdorff
area density admits a convex extension to $\Lambda^2 V$.
\end{theorem}

Hence the area is semi-elliptic over $\Z$, that is, planar discs
minimize the area among orientable surfaces with the same boundary.
Furthermore, an easy analysis of the proof shows that it 
also works in the non-orientable case:

\begin{theorem}\label{Z2}
In every finite-dimensional normed space $V$, the two-dimensional
Busemann--Hausdorff area density is semi-elliptic over $\Z_2$.
That is, every two-dimensional affine disc in $V$ minimizes
the Busemann--Hausdorff area among all
compact Lipschitz surfaces with the same boundary.
\end{theorem}

\begin{remark}
Since we do not assume that the norm of $V$ is strictly convex,
the area functional is not elliptic in general.
However it is easy to show that the theorems imply
that the Busemann--Hausdorff area is elliptic
if the norm is strictly (quadratically) convex.
\end{remark}

\begin{remark}
For the Holmes--Thompson surface area,
it had been noticed by Busemann, Ewald and Shephard \cite{BES}
that the density fails to be convex
already for the two-dimensional surface area for a certain norm on $\R^4$. 
Hence it is not elliptic over $\R$ (\cite{BI04}, see also \cite{BI-annals} for explicit examples.)
It turns out however that  discs in affine 2-planes minimize the Homes--Thompson 
area among all surfaces (with the same boundary) parametrized by topological discs 
(\cite{BI-annals}). The problem of whether the Holmes--Thomson area
is semi-elliptic over $\Z$ (that is, for competing surfaces that may have handles)
and beyond dimension~2 remains widely open and intriguing.
\end{remark}

The rest of the paper is organized as follows.
Sections \ref{sec-main} and \ref{sec-Z2} are devoted to proofs 
of Theorems \ref{main} and \ref{Z2} respectively.
The proof of the main Theorem \ref{main} goes via constructing
calibrating forms with constant coefficients, and the construction of the forms is based on 
a certain inequality for the Euclidean area of a convex centrally-symmetric polygon. As it was
pointed out by the anonymous referees, the key Proposition \ref{main-prop} is connected to 
investigations related to areas of random triangles carried out 
by Blaschke and others in response to Sylvester's Four-Point Problem. 
Namely, the proposition actually describes the probability 
measure supported on a centrally-symmetric convex curve so as to maximize 
the expectation of the area 
of a triangle formed by the center and two random points on the curve.

In
Section \ref{sec-kdim} we discuss prospectives and limitations of our methods for feasible generalizations
to higher dimensions.

\section{Proof of Theorem \ref{main}}
\label{sec-main}

\begin{definition}
Let $V$ be a finite-dimensional vector space,
$A\co GC_k(V)\to\R_+$ a $k$-dimensional density,
and $P\subset V$ a $k$-dimensional linear subspace. 
A \textit{calibrator} (or a \textit{calibrating form}) for $P$ with respect to $A$ is
an exterior $k$-form $\omega\in\Lambda^k V^*$ such that
for every simple $k$-vector $\sigma\in GC_k(V)$, one has $|\omega(\sigma)|\le A(\sigma)$,
and this inequality turns into equality if $\sigma\in\Lambda^k P$.
\end{definition}

One easily sees that $A$ admits a convex
extension to $\Lambda^k V$ if and only if
every 2-plane $P$ admits a calibrator.
(Indeed, an exterior $k$-form is a linear function
on $\Lambda^kV$, and a calibrator is a linear support for $A$.)
Hence our plan is to give an explicit construction of such calibrators
for $k=2$ and $A=A^{bh}$.

Let $V$ be a finite-dimensional normed space and $B$ its unit ball.
By means of approximation, it suffices to prove Theorem \ref{main}
in the case when $B$ is a polyhedron.
Fix a two-dimensional linear subspace $P\subset V$.
Our goal is to construct a calibrator for $P$
with respect to the Busemann--Hausdorff area density.

Consider the intersection $B\cap P$. It is a centrally symmetric polygon
$a_1a_2\dots a_{2n}$ (whose center is the origin $0$ of $V$). 
For each $i=1,\dots,n$, let $F_i\co V\to\R$
be a supporting linear function to $B$ such that $F_i=1$
on the segment $[a_ia_{i+1}]$.
For each $i=1,\dots,n$, let
$$
 p_i = \frac{2A(\triangle 0a_ia_{i+1})}{A(B\cap P)}
$$
where $A$ is an (arbitrarily normalized) area form on $P$
and $\triangle 0a_ia_{i+1}$ is the triangle with vertices $0,a_i,a_{i+1}$.
Note that $\sum p_i=1$.
Define a 2-form $\omega\in\Lambda^2 V^*$ by
$$
 \omega = \pi \cdot \!\!\! \sum_{1\le i<j\le n} p_ip_j\, F_i\wedge F_j .
$$
We are going to prove that $\omega$ is a desired calibrator for $P$.

Consider a simple 2-vector $\sigma=v_1\wedge v_2$ where
$v_1,v_2\in V$ are linearly independent vectors.
We need to prove that
$$
  |\omega(v_1\wedge v_2)| \le A^{bh}(v_1\wedge v_2)
$$
with equality in the case when $v_1,v_2\in P$.
Identify the plane $(v_1,v_2)$ with $\R^2$ by
means of the linear embedding $I\co \R^2\to V$
that takes the standard basis of $\R^2$ to $v_1$ and $v_2$.
Then
$$
 |\omega(v_1\wedge v_2)| = |I^*\omega| = 
 \pi \cdot \biggl| \sum_{1\le i<j\le n} p_ip_j\, f_i\wedge f_j \biggr|
$$
where $f_i=I^*F_i=F_i\circ I$ and the norms of 2-forms
are taken with respect to the Euclidean structure of $\R^2$.
The fact that $F_i$ is a supporting
function for $B$ implies that $f_i\le 1$ on the set $K=I^{-1}(B)$
which corresponds to the unit ball of the norm restricted to 
the plane $(v_1,v_2)$.
By the definition of the Busemann--Hausdorff area,
$$
 A^{bh}(v_1\wedge v_2) = \frac \pi{A(K)}
$$
where $A$ is the Euclidean area.
Thus the problem reduces to the following statement
from convex geometry on the plane.

\begin{proposition}\label{main-prop}
Let $K\subset\R^2$ be a symmetric convex polygon,
$f_1,\dots,f_n\co\R^2\to\R$ are linear functions such that
$f_i|_K\le 1$ for all $i$,
and $p_1,\dots,p_n$ are nonnegative real numbers such that $\sum p_i=1$.
Then
$$
 \biggl| \sum_{1\le i<j\le n} p_ip_j\, f_i\wedge f_j \biggr|
 \le \sum_{1\le i<j\le n} p_ip_j\, |f_i\wedge f_j|
 \le \frac1{A(K)} .
$$
In addition, if $K$ is a convex $2n$-gon $a_1a_2\dots a_{2n}$,
$f_i$ are supporting functions of $K$ corresponding to its sides
(that is, such that $f_i=1$ on $[a_ia_{i+1}]$),
and $p_i = 2A(\triangle 0a_ia_{i+1})/A(K)$, then the above inequalities
turn into equalities.
\end{proposition}

The proof of Proposition \ref{main-prop} occupies the rest of this section.
It consists of several elementary lemmas.

\begin{lemma}\label{lemma1}
Let $K=a_1a_2\dots a_{2n}$ be a symmetric $2n$-gon in $\R^2$.
Let $v_i = \overrightarrow {a_ia_{i+1}}$ for $i=1,\dots,n$. Then
$$
 A(K) = \sum_{1\le i<j\le n} |v_i\wedge v_j|
 = \biggl| \sum_{1\le i<j\le n} v_i\wedge v_j \biggr| .
$$
\end{lemma}

\begin{proof}
The second identity follows from the fact that all pairs $(v_i,v_j)$,
${1\le i<j\le n}$, are of the same orientation.
To prove the first one, observe that
$$
 A(K) = 2 A (a_1a_2\dots a_{n+1}) = 2\sum_{j=2}^{n} A(\triangle a_1a_ja_{j+1})
$$
since $K$ is symmetric. Further,
$$
 A(\triangle a_1a_ia_{i+1}) = \frac12 |\overrightarrow{a_1a_j}\wedge\overrightarrow{a_ja_{j+1}}|
  = \frac12 \sum_{i=1}^j |v_i\wedge v_j|
$$
since $\overrightarrow{a_1a_j}=v_1+v_2+\dots+v_{j-1}$ and all pairs $(v_i,v_j)$, $i<j$,
are of the same orientation. Plugging the second identity into the first one yields the result.
\end{proof}

The following lemma takes care of the equality case in Proposition \ref{main-prop}.

\begin{lemma}\label{lemma2}
Let $K=a_1a_2\dots a_{2n}$ be a symmetric $2n$-gon in $\R^2$.
For each $i=1,\dots,n$, let $v_i=\overrightarrow{a_ia_{i+1}}$,
$f_i\co\R^2\to\R$ the linear function
such that $f_i=1$ on $[a_ia_{i+1}]$ and $p_i = 2A(\triangle 0a_ia_{i+1})/A(K)$.
Then
$$
p_ip_j |f_i\wedge f_j| = \frac1{A(K)^2} |v_i\wedge v_j| .
$$
for all $i,j$, and therefore
$$
 \biggl|\sum_{1\le i<j\le n} p_ip_j\, f_i\wedge f_j \biggr|
 = \sum_{1\le i<j\le n} p_ip_j\, |f_i\wedge f_j|
 = \frac1{A(K)} .
$$
\end{lemma}

\begin{proof}
Denote $S_i=2A(\triangle 0a_ia_{i+1})=|a_i\wedge a_{i+1}|$,
then $p_i=S_i/A(K)$.
The oriented area form of $\R^2$ determines a linear isometry
$J\co(\R^2)^*\to\R^2$ in the standard way, and one easily sees
that $J(S_if_i) =\pm v_i$ (depending on the orientation).
Hence 
$$
S_iS_j |f_i\wedge f_j | =  |(S_if_i)\wedge (S_jf_j) | =   |v_i\wedge v_j |
$$
and therefore
$$
 p_ip_j |f_i\wedge f_j| = \frac1{A(K)^2} S_iS_j |f_i\wedge f_j|  =  \frac1{A(K)^2} |v_i\wedge v_j| .
$$

To prove the second assertion, 
observe that all pairs $(f_i,f_j)$, $1\le i<j\le n$, are of the same orientation, hence
$$
 \biggl|\sum_{1\le i<j\le n} p_ip_j\, f_i\wedge f_j \biggr| = \sum_{1\le i<j\le n} p_ip_j |f_i\wedge f_j |
  = \frac1{A(K)^2} \sum_{1\le i<j\le n} |v_i\wedge v_j | = \frac1{A(K)}
$$
where the last identity follows from Lemma \ref{lemma1}.
\end{proof}

It remains to prove the inequality part of Proposition \ref{main-prop}.
The next lemma covers the principal case when the $f_i$'s are
supporting functions of the sides.

\begin{lemma}\label{lemma3}
Let $K=a_1a_2\dots a_{2n}$ be a symmetric $2n$-gon in $\R^2$.
For each $i=1,\dots,n$, let $f_i\co\R^2\to\R$ be linear function
such that $f_i=1$ on $[a_ia_{i+1}|$.
Let $p_1,\dots,p_n$ be nonnegative real numbers such that $\sum p_i=1$.
Then
\begin{equation}
\label{main-ineq}
 \sum_{1\le i<j\le n} p_ip_j\, |f_i\wedge f_j| \le \frac1{A(K)} .
\end{equation}
\end{lemma}

\begin{proof}
Denote $v_i=\overrightarrow{a_ia_{i+1}}$, $q_i = 2A(\triangle 0a_ia_{i+1})/A(K)$
and $\lambda_i=p_i/q_i$.
By the first part of Lemma \ref{lemma3},
$$
q_iq_j |f_i\wedge f_j| = \frac1{A(K)^2} |v_i\wedge v_j| .
$$
Let $v'_i=\lambda_i v_i$ for $i=1,\dots,n$.
Consider a symmetric $2n$-gon $K'=a'_1\dots a'_{2n}$ such that
$\overrightarrow{a'_ia'_{i+1}}=v'_i$. 
Then
\begin{equation*}
\begin{aligned}
\sum_{1\le i<j\le n} p_ip_j\, |f_i\wedge f_j | 
&= \sum_{1\le i<j\le n} \lambda_i\lambda_j q_iq_j\, |f_i\wedge f_j |  \\
&= \frac1{A(K)^2} \sum_{1\le i<j\le n} \lambda_i\lambda_j \, |v_i\wedge v_j | \\
&= \frac1{A(K)^2} \sum_{1\le i<j\le n} |v'_i\wedge v'_j | 
= \frac{A(K')}{A(K)^2}
\end{aligned} 
\end{equation*}
where the last identity follows from Lemma \ref{lemma1}.
Therefore
\begin{equation}
\label{sum-value}
\sum_{1\le i<j\le n} p_ip_j\, |f_i\wedge f_j | = \frac{A(K')}{A(K)^2} .
\end{equation}

Denote $\ell_i=|v_i|$ and $\ell'_i=|v'_i|=\lambda_i|v_i|$.
Let $h_i$  denote the distance from the origin
to the line containing the side $[a_ia_{i+1}]$.
Then $2A(\triangle 0a_ia_{i+1})=h_i\ell_i$, hence
$q_i = h_i\ell_i/A(K)$.
Therefore
$$
 1 = \sum p_i = \sum \lambda_iq_i = \frac1{A(K)}\sum \lambda_i h_i\ell_i
 =  \frac1{A(K)}\sum h_i\ell'_i .
$$
The last sum is the two-dimensional mixed volume $V(K,K')$, 
thus $V(K,K')=A(K)$.
By the Minkowski inequality (which is a special case of the Alexandrov--Fenchel
inequality), we have $V(K,K')^2\ge A(K)A(K')$.
Therefore $A(K')\le A(K)$. Substituting this inequality into \eqref{sum-value}
yields the assertion of the lemma.
\end{proof}

To complete the proof of Proposition \ref{main-prop} it remains
to prove the inequality \eqref{main-ineq} in a slightly more general setting,
namely when $K$ is a symmetric polygon (not necessarily with $2n$ sides)
and $f_1,\dots,f_n$ are arbitrary linear functions such that $f_i|_K\le 1$.
The condition $f_i|_K\le 1$ means that $f_i$ belongs to the
polar polygon $K^*\subset(\R^2)^*$. Consider the left-hand side
of \eqref{main-ineq} as a function in one variable $f_i$ with others staying fixed.
This function is convex (since it is a sum of the absolute values of linear functions),
therefore it attains its maximum on $K^*$ at a vertex of~$K^*$.
The vertices of $K^*$ are supporting linear functions to $K$
at its sides. So it suffices to consider the case when
each $f_i$ equals 1 on one of the sides of $K$.

If two of the functions $f_i$ and $f_j$ coincide (without loss of generality, $f_1=f_n$),
one reduces the problem to a smaller number of functions as follows:
drop $f_n$ from the list of functions
and replace $p_1,p_2,\dots,p_n$ by $p_1+p_n,p_2,\dots,p_{n-1}$.
Also note that changing sign of one of the functions $f_i$ does not change
the left-hand side of \eqref{main-ineq}.
Thus it suffices to consider the case when $\pm f_1,\dots,\pm f_n$ are all distinct.
If $n=1$, the left-hand side of \eqref{main-ineq} is zero so the
inequality is trivial.
If $n>1$, applying Lemma \ref{lemma3} to the polygon $K' = \bigcap_{i=1}^n \{x:|f_i(x)|\le 1\}$
yields that
$$
\sum_{1\le i<j\le n} p_ip_j\, |f_i\wedge f_j| \le \frac1{A(K')} \le \frac1{A(K)}
$$
since $K\subset K'$. This completes the proof of Proposition \ref{main-prop}
and hence the proof of Theorem \ref{main}.

\section{Proof of Theorem \ref{Z2}}
\label{sec-Z2}

Let $V$ be a finite-dimensional normed space and $B$ its unit ball.
Let $M$ be a compact two-dimensional smooth manifold with $\pd M\simeq S^1$
and $S\co M\to V$ a Lipschitz map such that $S|_{\pd M}$ parametrizes
the boundary of a 2-disc $D$ lying in a two-dimensional linear subspace $P\subset V$.
Our goal is to prove that $A^{bh}(S)\ge A^{bh}(D)$.
By means of approximation, we may assume that $B$ is a polyhedron.

The intersection $B\cap P$ is a symmetric polygon
$a_1a_2\dots a_{2n}$.
Define linear functions $F_i\co V\to\R$ and coefficients $p_i$
($i=1,\dots,n$) as in the proof of Theorem~\ref{main}.
Define a function $\alpha\co GC_2(V)\to\R_+$ by
$$
 \alpha (\sigma) = \pi \cdot \!\!\! \sum_{1\le i<j\le n} p_ip_j\, |(F_i\wedge F_j)\cdot\sigma|
$$
for all $\sigma\in GC_2(V)$.
Similarly to the proof of Theorem \ref{main},
Proposition \ref{main-prop} implies that
\begin{equation}
\label{Z2-calibrator}
\alpha(\sigma)\le A^{bh}(\sigma)
\end{equation}
for all $\sigma\in GC_2(V)$, and this inequality turns into equality if $\sigma\in\Lambda^2(P)$.
Indeed, let $\sigma = v_1\wedge v_2$ where $v_1,v_2\in V$ are linearly independent vectors
and consider the linear embedding $I\co\R^2\to V$
that takes the standard basis of $\R^2$ to $v_1$ and $v_2$.
Let $K=I^{-1}(B)$, then
$$
 A^{bh}(\sigma) = \frac \pi{A(K)}
$$
where $A$ is the Euclidean area.
On the other hand
$$
 \alpha(\sigma) = 
 \pi \cdot \sum_{1\le i<j\le n} p_ip_j\, |f_i\wedge f_j|
$$
where $f_i=F_i\circ I$.
Recall that $F_i|_B\le 1$, hence $f_i|_K\le 1$ for all $i$.
By Proposition \ref{main-prop}, we have
$$
 \sum_{1\le i<j\le n} p_ip_j\, |f_i\wedge f_j|
 \le \frac1{A(K)}
$$
and \eqref{Z2-calibrator} follows.
In the case when $v_1,v_2\in P$, the equality case of
Proposition \ref{main-prop} and the construction of $f_i$ and $p_i$
yields the equality in \eqref{Z2-calibrator}.

It remains to show that \eqref{Z2-calibrator} implies the inequality
$A^{bh}(S)\ge A^{bh}(D)$.
For each pair $i,j$, $1\le i<j\le n$, define a linear map
$F_{ij}\co V\to\R^2$ by
$$
 F_{ij}(v) = (F_i(v),F_j(v))\in\R^2, \qquad v\in V,
$$
and consider the map $F_{ij}\circ S\co M\to\R^2$.
The Euclidean area $A(F_{ij}\circ S)$ of this map is given by
$$
 A(F_{ij}\circ S) = \int_S |F_i\wedge F_j| .
$$
(The term $|F_i\wedge F_j|$ here is a two-dimensional density in $V$
given by $\sigma\mapsto|(F_i\wedge F_j)\cdot\sigma|$, $\sigma\in GC_2(V)$.
Recall that the integration of a density over a surface, orientable or not,
is well-defined.)
Therefore
$$
 \int_S \alpha = \int_S \sum_{1\le i<j\le n} p_ip_j|F_i\wedge F_j|
 = \sum_{1\le i<j\le n} p_ip_j A(F_{ij}\circ S) .
$$
Similarly, for the planar disc $D\subset P$ we have
$$
 \int_D \alpha = \int_D \sum_{1\le i<j\le n} p_ip_j|F_i\wedge F_j|
 = \sum_{1\le i<j\le n} p_ip_j A(F_{ij}(D)) .
$$
Observe that $F_{ij}(D)\subset F_{ij}\circ S(M)$ because $S|_{\pd M}$
is a degree 1 map from $\pd M$ to $\pd D$. Therefore
$A(F_{ij}\circ S) \ge A(F_{ij}(D))$ for all $i,j$, and hence the above
identities imply that $\int_S \alpha \ge \int_D \alpha$.
By \eqref{Z2-calibrator}, we have $A^{bh}(S)=\int_S A^{bh} \ge \int_S \alpha$,
and $A^{bh}(D) = \int_D A^{bh} \ge \int_D \alpha$ by the equality case
of \eqref{Z2-calibrator}. Thus
$$
A^{bh}(S)\ge \int_S \alpha\ge \int_D \alpha=A^{bh}(D)
$$
and Theorem \ref{Z2} follows.

\section{Remarks on the higher-dimensional case}
\label{sec-kdim}

Although we cannot generalize Theorem \ref{main}
to surfaces of dimension $k>2$ at the moment, some of the
arguments from Section \ref{sec-main} apply in this case as well.
Moreover, the convexity of a $k$-dimensional Busemann--Hausdorff 
surface area density is equivalent to a $k$-dimensional
analogue of Proposition \ref{main-prop}.
In hope that this approach will be useful,
we formulate this equivalence as the following proposition.

\begin{proposition}
\label{prop-kdim}
For every positive integer $k$,
the following two assertions are equivalent.

(i) In every finite-dimensional normed space $V$,
the $k$-dimensional Busemann--Hausdorff area density
admits a convex extension to $\Lambda^k V$.

(ii) For every $n\ge k$ and every central
symmetric convex polyhedron $K\subset\R^k$
with $n$ pairs of opposite faces $F_1,F_1',\dots,F_n,F_n'$
there exist
a collection $\mu_{i_1i_2\dots i_k}$, $1\le i_1<\dots<i_k\le n$,
of real coefficients such that the following holds.
For every convex polyhedron $K'\in\R^k$ and every collection
of linear functions $f_1,\dots,f_n\co\R^k\to\R$ such that
$f_i|_{K'}\le 1$ for all $i$, one has
\begin{equation}
\label{ineq-kdim}
 \left| \sum \mu_{i_1i_2\dots i_k} f_{i_1}\wedge\dots\wedge f_{i_k}\right| \le \frac1{\vol(K')},
\end{equation}
and this inequality turns into equality if $K'=K$
and $f_i$'s are supporting linear functions
corresponding to faces $F_i$'s (i.e., $f_i|_{F_i}=1$).
\end{proposition}

\begin{remark}
In the two-dimensional case, we just defined $\mu_{ij}=p_ip_j$
where $p_i$ is the portion of the area of $K$ spanned by
the $i$th pair of its faces. This construction is not unique:
for many polygons $K$, other choice of $\mu_{ij}$
works as well. In higher dimensions, we have no idea
how a suitable collection of coefficients
(depending on a polyhedron)
could be defined, and the most straightforward generalization
of the two-dimensional construction does not work.
\end{remark}

\begin{proof}[Proof of Proposition \ref{prop-kdim}]
The implication (ii)$\Rightarrow$(i) is similar to the
deduction of Theorem \ref{main} from Proposition \ref{main-prop}
in Section \ref{sec-main}. To see that (i) implies (ii),
consider a polyhedron $K$ as in (ii) and equip $\R^k$
with a norm $\|\cdot\|$ whose unit ball is $K$.
Let $f_1^K,\dots,f_n^K\co\R^k\to\R$ be the linear functions
corresponding to the faces of $K$, then
$$
 K = \{ x\in\R^k : |f_i^K(x)|\le 1, \ i=1,\dots,n \} .
$$
Hence the linear map $f^K\co\R^k\to\R^n$ given by
$$
 f^K(x) = (f^K_1(x),\dots,f^K_n(x)), \qquad x\in\R^k,
$$
is an isometric embedding of the normed space $(\R^k,\|\cdot\|)$
into $\R^n_\infty=(\R^n,\|\cdot\|_\infty)$ where
the norm $\|\cdot\|_\infty$ on $\R^n$ is defined by
$$
 \|x\|_\infty = \max_{1\le i\le n} |x_i|.
$$
Assuming (i), the $k$-dimensional Busemann--Hausdorff area density
in $\R^n_\infty$ admits a convex extension, and therefore
there exists a calibrating form $\omega\in\Lambda^k(\R^n)^*$
for the linear subspace $f^K(\R^k)$.
The fact that $\omega$ is a calibrator means that
for every linear map $f\co\R^k\to\R^n$ one has
\begin{equation}
\label{calibrate-kdim}
 |f^*\omega| \le \ep_k \vol(f^{-1}(B))^{-1}
\end{equation}
with the equality for $f=f^K$,
where $\ep_k$ is the volume of the Euclidean unit ball in $\R^k$
and $B=[-1,1]^n$ is the unit ball of $\R^n_\infty$.
Let $\mu_{i_1i_2\dots i_k}$, $1\le i_1<\dots<i_k\le n$,
 be the coefficients of $\ep_k^{-1}\omega$,
that is,
$$
 \omega = \ep_k \sum \mu_{i_1i_2\dots i_k} dx_{i_1}\wedge dx_{i_2}\wedge\dots\wedge dx_{i_k} .
$$
Given a polyhedron $K'\subset\R^k$ and linear functions $f_i$ as in (ii),
consider $f=(f_1,\dots,f_n)\co\R^k\to\R^n$. Then
$$
 f^*\omega = \ep_k \sum \mu_{i_1i_2\dots i_k} f_{i_1}\wedge\dots\wedge f_{i_k}
$$
and $\vol(f^{-1}(B))^{-1}\le \vol(K')^{-1}$ 
since
$$
 f^{-1}(B) = \{ x\in\R^k : |f_i(x)| \le 1, i=1,\dots,n\} \supset K'
$$
Thus \eqref{calibrate-kdim} implies \eqref{ineq-kdim},
and the equality in \eqref{calibrate-kdim} for $f=f^K$
yields the equality in \eqref{ineq-kdim} for $K'=K$ and $f_i=f_i^K$.
\end{proof}

\end{document}